\documentclass[a4paper]{amsart}
\usepackage[english]{babel}
\usepackage[utf8]{inputenc}
\usepackage{lmodern}
\usepackage[T1]{fontenc}

\usepackage{lipsum}
\usepackage{amssymb, amsthm, amsmath}
\usepackage{enumitem}
\usepackage{bbm,dsfont}
\usepackage[all]{xy}

\newcommand*\Laplace{\mathop{}\!\mathbin\bigtriangleup}

\newcommand{\R}{\mathbb{R}}
\newcommand{\C}{\mathbb{C}}

\newcommand{\A}{\mathcal{A}}
\newcommand{\U}{{U}}

\renewcommand{\L}{\mathcal{L}}
\newcommand{\T}{\mathcal{T}}

\renewcommand{\d}{\mathrm{d}}
\renewcommand{\Re}{\operatorname{Re}}

\newcommand{\fra}{\mathfrak{a}}

\renewcommand{\mid}{\, \vert \,}
\renewcommand{\MR}{\textit{MR}\,}

\newcommand\lra{\longrightarrow}
\newcommand\restr[2]{{#1}{_{|#2}}}
\theoremstyle{plain}
\newtheorem{theorem}{Theorem}[section]
\newtheorem{proposition}[theorem]{Proposition}

\theoremstyle{definition}
\newtheorem{definition}[theorem]{Definition}

\newtheorem{remark}[theorem]{Remark}

\begin{document}
	\title{On the norm-continuity for evolution family arising from non-autonomous forms$^*$\footnote{ $^*$This work is supported by Deutsche Forschungsgemeinschaft  DFG (Grant LA 4197/8-1)}}
	\author{
		Omar EL-Mennaoui and  Hafida Laasri \\ 
	}
	
	\address{University of Wuppertal
School of Mathematics and Natural Sciences
Arbeitsgruppe Funktionalanalysis, Gau\ss stra\ss e 20
D-42119 Wuppertal, Germany}
	\email{laasri@uni-wuppertal.de}
	\address{
Ibn Zohr University, Faculty of Sciences
Departement of mathematics, Agadir, Morocco}
	\email{elmennaouiomar@yahoo.fr}
	
	\maketitle

\begin{abstract} We consider evolution equations of the form
\begin{equation*}\label{Abstract equation}
\dot u(t)+\A(t)u(t)=0,\ \ t\in[0,T],\ \
u(0)=u_0,
\end{equation*}
where $\A(t),\ t\in [0,T],$ are associated with a non-autonomous sesquilinear 
form $\fra(t,\cdot,\cdot)$ on a Hilbert space $H$ with constant domain $V\subset H.$ 
In this note we continue the study of fundamental operator theoretical properties of the solutions. 
We give a sufficient condition for norm-continuity of evolution families on each spaces $V, H$ and 
on the dual space $V'$ of $V.$ The abstract results are applied to a class of equations governed by
time dependent Robin boundary conditions on exterior domains and by Schrödinger operator with 
time dependent potentials. 
\end{abstract}



\section*{Introduction\label{s1}}
 \noindent Throughout this paper $H,V$ are two separable Hilbert spaces over $\mathbb C$ such that $V$ is 
 densely and continuously embedded into $H$ (we write $V \underset d \hookrightarrow H$).
   We denote by $(\cdot \mid \cdot)_V$ the scalar product and $\|\cdot\|_V$ the norm
on $V$ and by $(\cdot \mid \cdot)_H, \|\cdot\|_H$ the corresponding quantities in $H.$ Let $V'$ 
be the antidual of $V$ and denote by $\langle ., . \rangle$ the duality
between $V'$ and $V.$ As usual, by identifying $H$ with  $H'$  
we have $V\underset d\hookrightarrow H\cong H'\underset d\hookrightarrow V'$
see e.g., \cite{Bre11}. 
\par\noindent Let $\fra:[0,T]\times V \times V \to \C$ be a \textit{non-autonomous sesquilinear form}, i.e., $\fra(t;\cdot,\cdot)$ is for each $t\in[0,T]$ a sesquilinear form, 
\begin{equation}\label{measurability} 
\fra(\cdot;u,v) \text{ is measurable for all } u,v\in V, 
 \end{equation}
such that 
\begin{equation}\label{eq:continuity-nonaut}
|\fra(t,u,v)| \le M \Vert u\Vert_V \Vert v\Vert_V\ \text{ and } \Re ~\fra (t,u,u)\ge \alpha \|u\|^2_V  \quad \ (t,s\in[0,T], u,v\in V),
\end{equation}
for some  constants $M, \alpha>0$ that are independent of $t, u,v.$ Under these assumptions
there exists  for each $t\in[0,T]$ an isomorphism 
$\A(t):V\to V^\prime$ such that
$\langle \A(t) u, v \rangle = \fra(t,u,v)$ for all $u,v \in V.$ It is well known that  $-\A(t),$ seen as unbounded operator 
with domain $V,$ generates an analytic $C_0$-semigroup on $V'$. The operator $\A(t)$ is usually called the operator
associated with $\fra(t,\cdot,\cdot)$ on $V^\prime.$ Moreover, we associate an operator $A(t)$ with $\fra(t;\cdot,\cdot)$ on $H$ 
as follows
\begin{align*}
D(A(t))={}&\{u\in V \mid \exists f\in H \text{ such that } \fra(t;u,v)=(f\mid v)_H  \text{ for all } v\in V \}\\
A(t) u = {}&  f.
\end{align*}
It is not difficult to see that  $A(t)$ is the part of $\A(t)$ in $H.$ In fact, we have
$D(A(t))= \{ u\in V : \A(t) u \in H \}$ and 
$A(t) u =  \A(t) u.$
Furthermore, $-A(t)$ with domain $D(A(t))$ generates
a holomorphic $C_0$-semigroup on $H$ which is the restriction to $H$ of that generated by $-\A(t).$ For all this results 
see e.g. \cite[Chapter 2]{Tan79} or \cite[Lecture 7]{Ar06}. 

\par\noindent We now assume that there exist  $0< \gamma<1$ and a continuous function $\omega:[0,T]\longrightarrow [0,+\infty)$ such that 
  \begin{equation}\label{eq 1:Dini-condition} 
  |\fra(t,u,v)-\fra(s,u,v)| \le\omega(|t-s|) \Vert u\Vert_{V_\gamma} \Vert v\Vert_{V_\gamma}\quad \ (t,s\in[0,T], u,v\in V),
  \end{equation}
  with
 \begin{equation}\label{eq 2:Dini-condition}\sup_{t\in[0,T]} \frac{\omega(t)}{t^{\gamma/2}}<\infty \quad \text{ and }
 \int_0^T\frac{\omega(t)}{t^{1+\gamma/2}}<\infty
 \end{equation}
 where $V_\gamma:=[H,V]$ is the complex interpolation space. In addition we assume that $\fra(t_0; \cdot,\cdot)$ 
 has the square root property for some $t_0\in[0,T]$ (and then for all $t\in[0,T]$ by \cite[Proposition 2.5]{Ar-Mo15} ), 
 i.e.,
 
 \begin{equation}\label{square property}
D(A^{\frac{1}{2}}(t_0))=V.
 \end{equation} 
Recall that  for  symmetric forms, i.e., if $\fra(t;u,v)=\overline{\fra(t;v,u)}$ for all $t,u,v,$ then the square root property is satisfied. 

Under the assumptions \eqref{measurability}-\eqref{square property} it is known that for each $x_0\in V$ the non-autonomous homogeneous Cauchy problem
  \begin{equation}\label{Abstract Cauchy problem 0}
  \left\{
\begin{aligned}
\dot u(t)&+\A(t)u(t)= 0\quad \hbox{a.e. on}\ [0,T],\\
u(0)&=x_0, \,\\
\end{aligned}
\right.
\end{equation}
 has a unique solution $u \in \MR(V,H):=L^2(0,T;V)\cap H^1(0,T;H)$ such that $u\in C([0,T];V).$ This result has been proved by Arendt and Monniaux \cite{Ar-Mo15} (see also \cite{ELLA15}) when the form $\fra$ satisfies the weaker condition
  \begin{equation}\label{eq 1:Dini-conditionWeaker} 
  |\fra(t,u,v)-\fra(s,u,v)| \le\omega(|t-s|) \Vert u\Vert_V \Vert v\Vert_{V_\gamma}\quad \ (t,s\in[0,T], u,v\in V).
  \end{equation}

  \medskip

 \par\noindent In this paper we continue to investigate further regularity of the solution of 
 (\ref{Abstract Cauchy problem 0}).  For this it is necessary to associate to the Cauchy problem \eqref{Abstract Cauchy problem 0} an \textit{evolution family} \[\mathcal U:=\Big\{U(t,s):\  0\leq s\leq t\leq T\Big\}\subset \L(H)\] which means that: 
 \begin{itemize}
 	\item[$(i)$] $U(t,t)=I$ and $U(t,s)= U(t,r)U(r,s)$ for every $0\leq r\leq s\leq t\leq T,$ 
 	\item [$(ii)$] for every $x\in X$ the function $(t,s)\mapsto U(t,s)x$ is continuous into $H$ 
 	for $0\leq s\leq t\leq T$
 	\item [$(iii)$] for each $x_0\in H, U(\cdot,s)x_0$ is the unique solution of (\ref{Abstract Cauchy problem 0}).
 \end{itemize}
\begin{definition} Let $Y\subseteq H$ be a subspace. An evolution family $\mathcal U\subset \L(H)$ is said to be \textit{norm continuous in $Y$} if $\mathcal U\subset\L(Y)$ and the map $(t,s)\mapsto U(t,s)$ is
norm continuous with value
in $\L(Y)$ for $0\leq s<t\leq T.$
\end{definition}
 \par\noindent If the non-autonomous form $\fra$ satisfies the weaker condition (\ref{eq 1:Dini-conditionWeaker}) then it is known that (\ref{Abstract Cauchy problem 0}) is governed by an evolution family which is norm continuous in $V$ \cite[Theorem 2.6]{LH17}, and norm continuous in $H$ if in addition $V\hookrightarrow H$ is compact \cite[Theorem 3.4]{LH17}. However, for many boundary value problem the compactness embedding fails. 
 \par\noindent
In this paper we prove that the compactness assumption can be omitted provided $\fra$ satisfies
\eqref{eq 1:Dini-condition} instead of (\ref{eq 1:Dini-conditionWeaker}). This will allow us to consider a large class of examples of applications. 
 One of the main ingredient used here is the non-autonomous 
\textit{returned adjoint form}  $\fra^*_r : [0,T]\times V\times V\lra \C$ defined  by 
\begin{equation}\label{returnd adjoint form} 
\fra^*_r(t,u,v):=\overline{\fra(T-t,v,u)} \quad \ (t,s\in[0,T], u,v\in V).
\end{equation}
The concept of  returned adjoint forms appeared in the work of D. Daners \cite{Daners01} but for different interest. Furthermore,  \cite[Theorem 2.6]{LH17} cited above will be also needed to prove our main result. 
\medskip

 \par\noindent We note that the study of regularity properties of the evolution family 
with respect to $(t,s)$ in general Banach spaces has been investigated in the case of constant domains by Komatsu \cite{H.Ko61}  
and Lunardi \cite{Lu89}, and by Acquistapace \cite{Ac88} for time-dependent domains. 
\medskip

We illustrate our abstract results by two relevant examples. The first one concerns the Laplacian with non-autonomous Robin boundary conditions on a unbounded Lipschitz domain. The second one traits a class of  Schr\"odinger operators with time dependent potential. 
  
\section{Preliminary results \label{Approximation}}
 Let $
\fra: [0,T]\times V\times V \to \C$ a non-autonomous sesquilinear form satisfying \eqref{measurability} and (\ref{eq:continuity-nonaut}). Then the following well known result regarding \textit{$L^2$-maximal regularity in $V'$} 
is due to J. L. Lions

\begin{theorem}[Lions, 1961]\label{wellposedness in V'2}
	For each given $s\in[0,T)$ and $x_0\in H$ the homogenuous Cauchy problems 
\begin{equation}\label{evolution equation u(s)=x}
\left\{
\begin{aligned}
\dot u(t)&+\A(t)u(t)= 0\quad \hbox{a.e. on}\ [s,T],\\
u(s)&=x, \,\\
\end{aligned}
\right.
\end{equation} has a unique solution $u \in \MR(V,V'):=\MR(s,T;V,V'):=L^2(s,T;V)\cap H^1(s,T;V').$ 
\end{theorem}
Recall that the maximal regularity space $\MR(V,V')$ is continuously embedded into $C([s,T],H)$ \cite[page 106]{Sho97}. A proof of Theorem \ref{wellposedness in V'2} using a representation theorem of linear functionals, known in the literature as \textit{Lions's representation Theorem}
can be found in \cite[page 112]{Sho97} or \cite[XVIII, Chpater 3, page 513]{DL88}.

\par\noindent 
\par\noindent
Furthermore, we consider the non-autonomous adjoint form $\fra^*:[0,T]\times V\times V\lra \C$ of $\fra$ defined by
\[ \fra^*(t;u,v):=\overline{\fra(t;v,u)}\]
for all $t\in[0,t]$ and $u,v\in V.$  Finally, we will need to consider
the returned adjoin form $\fra^*_r:[0,T]\times V\times V\lra \C$ given by \[\fra^*_r(t,u,v):=\fra^*(T-t,u,v).\]
Clearly, the adjoint form is a non-autonomous sesquilinear form and satisfies 
\eqref{measurability} and (\ref{eq:continuity-nonaut}) with the same constant $M, \alpha.$ 
Moreover, the adjoint operators $A^*(t), t\in[0,T]$ of $A(t), t\in[0,T]$ coincide with the 
operators associated with $\fra^*$ on $H.$ Thus applying Theorem \ref{wellposedness in V'2} to the returned adjoint form we obtain that 
the Cauchy problem associated with $\A^*_r(t):=\A^*(T-t)$ 
\begin{equation}\label{evolution equation u(s)=x returned}
\left\{
\begin{aligned}
\dot v(t)&+\A^*_r(t)v(t)= 0\quad \hbox{a.e. on}\ [s,T],\\
v(s)&=x, \,\\
\end{aligned}
\right.
\end{equation}
has for each $x\in H$ a unique solution $v\in MR(V,V').$ Accordingly,  for every 
$(t,s)\in\overline{\Delta}:=\{ (t,s)\in[0,T]^2:\  t\leq s\}$ 
and every $x\in H$ we can define the following family of linear operators
\begin{equation}\label{evolution family}
 U(t,s)x:=u(t) \quad \hbox{ and }\quad 
U^*_r(t,s)x:=v(t),
\end{equation}
where $u$ and $v$ are the unique solutions in $MR(V,V')$ 
respectively of (\ref{evolution equation u(s)=x}) and (\ref{evolution equation u(s)=x returned}). 
Thus each family $\{\mathcal U(t,s):\ (t,s)\in\overline{\Delta}\}$ 
and $\{\mathcal U^*_r(t,s):\ (t,s)\in\overline{\Delta}\}$ yields a contractive,
strongly continuous evolution family on $H$ \cite[Proposition]{LH17}.

 \par\noindent In the autonomous case, i.e., if $\fra(t,\cdot,\cdot)=\fra(\cdot,\cdot)$ for all $t\in[0,T],$ then one knows that $-A_0,$ the operator associated with $\fra_0$ in $H,$ generates a $C_0$-semigroup  $(T(t))_{t\geq 0}$ in $H.$ In this case $\U(t,s):=T(t-s)$ yields a strongly continuous evolution family on $H.$ Moreover, we have
\begin{equation}\label{equalities: adjoint EVF and EVF autonomous case}
\U(t,s)^\prime=T(t-s)^{\prime}=T^*(t-s)=\U^*(t,s)=\U^*_r(t,s).
\end{equation}
Here, $T(\cdot)^\prime$ denote the adjoint of $T(\cdot)$ which coincides with the $C_0$-semigroup  $(T^*(t))_{t\geq 0}$ associated with the adjoint form $\fra^*.$ 
In the non-autonomous setting however,  (\ref{equalities: adjoint EVF and EVF autonomous case}) fails in general even in the finite dimensional case, see \cite[Remark 2.7]{Daners01}. Nevertheless, Proposition \ref{equalities: adjoint EVF and EVF} below shows that the evolution families $\U$ and $\U_r^*$ can be related in a similar way. This formula appeared  in  \cite[Theorem 2.6]{Daners01}. 
\begin{proposition}\label{equalities: adjoint EVF and EVF} Let $\U$ and $\U^*_r$ be given by (\ref{evolution family}). Then we have 
	\begin{equation}\label{Key equalities: adjoint EVF and EVF}
	\big [ \U^*_r(t,s)\big ]^\prime x=\U(T-s,T-t)x
	\end{equation} 
for all $x\in H$ and $(t,s)\in\overline{\Delta}.$	
\end{proposition}
The equality \eqref{Key equalities: adjoint EVF and EVF} will play a crucial role in the proof of our main result.
We include here a new proof for the sake of completeness. 
\begin{proof}(of Proposition \ref{equalities: adjoint EVF and EVF})
	Let $\Lambda=(0=\lambda_0<\lambda_1<...<\lambda_{n+1}=T)$ be
	a subdivision of $[0,T].$ Let $\fra_k:V \times V \to \mathbb C\ \ \hbox{ for } k=0,1,...,n$ be given by \begin{equation*}\begin{aligned}
		\ \fra_k(u,v):=\fra_{k,\Lambda}(u,v):=\frac{1}{\lambda_{k+1}-\lambda_k}
		\int_{\lambda_k}^{\lambda_{k+1}}&\fra(r;u,v){\rm  d}r\ \hbox{ for } u,v\in V. \
	\end{aligned}
\end{equation*}
All these  forms satisfy (\ref{eq:continuity-nonaut}) with the same constants $\alpha, M.$ The associated operators in $V'$ are denoted by $\A_k\in \L(V,V')$ and are given  for all $u\in V$ and $k=0,1,...,n$ by
\begin{equation}\label{eq:op-moyen integrale}
\A_ku :=\A_{k,\Lambda}:=\frac{1}{\lambda_{k+1}-\lambda_k}
\int_{\lambda_k}^{\lambda_{k+1}}\A(r)u{\rm  d}r.\ \ \  
\end{equation}
Consider the non-autonomous form $\fra_\Lambda:[0,T]\times V \times V \to \C$ defined by
    \begin{equation}\label{form: approximation formula1}
    \fra_{\Lambda}(t;\cdot,\cdot):=\begin{cases}
    \fra_k(\cdot,\cdot)&\hbox{if }t\in [\lambda_k,\lambda_{k+1})\\
    \fra_n(\cdot,\cdot)&\hbox{if }t=T\ .
    \end{cases}
\end{equation}
Its associated time dependent operator $\A_\Lambda(\cdot):  [0,T]\subset \L(V,V')$ is given by 
      \begin{equation}\label{form: approximation formula1}
    \A_{\Lambda}(t):=\begin{cases}
    \A_k&\hbox{if }t\in [\lambda_k,\lambda_{k+1})\\
    \A_n &\hbox{if }t=T\ .
    \end{cases}
\end{equation}
Next denote by  $T_k$ the $C_0-$semigroup associated with $\fra_k$  in $H$ for all $k=0,1...n.$ Then applying Theorem \ref{wellposedness in V'2}) to the form $\fra_\Lambda$ we obtain that in this case the associated evolution family $\U_\Lambda(t,s)$ is given explicitly for $\lambda_{m-1}\leq s<\lambda_m<...<\lambda_{l-1}\leq t<\lambda_{l}$
	by 
	\begin{equation}\label{promenade1}\U_\Lambda (t,s):= T_{l-1}(t-\lambda_{l-1})
	T_{l-2}(\lambda_{l-1}-\lambda_{l-2})...T_{m}(\lambda_{m+1}-\lambda_{m})T_{m-1}(\lambda_{m}-s),
	\end{equation}
	and for $\lambda_{l-1}\leq a\leq b<\lambda_{l}$ by 
\begin{equation}\label{promenade2}U_\Lambda (t,s):= T_{l-1}(t-s).\end{equation}
By \cite[Theorem 3.2]{LASA14} we know that $(\U_\Lambda)_{\Lambda}$ converges weakly in $MR(V,V')$ as $|\Lambda|\to 0$ and 
\[\lim\limits_{|\Lambda|\to 0}\|\U_\Lambda-\U\|_{MR(V,V')}=0\]
The continuous embedding of $MR(V,V')$ into $C([0,T];H)$ implies that $\lim\limits_{|\Lambda|\to 0}\U_\Lambda=\U$
in the weak operator topology of $\L(H).$
\par\noindent Now, let $(t,s)\in \overline{\Delta}$ with $\lambda_{m-1}\leq s<\lambda_m<...<\lambda_{l-1}\leq t<\lambda_{l}$ be fixed. Applying the above approximation argument to $\fra_r^*$ one obtains that
\begin{align}\label{eq1 proof returned adjoint}
\U^*_{\Lambda,r}(t,s)&=T_{l-1,r}^*(t-\lambda_{l-1})
T_{l-2,r}^*(\lambda_{l-1}-\lambda_{l-2})...T_{m,r}^*(\lambda_{m+1}-\lambda_{m})T_{m-1,r}^*(\lambda_{m}-s)
\\\label{eq2 proof returned adjoint}&=T_{l-1,r}^\prime(t-\lambda_{l-1})
T_{l-2,r}^\prime(\lambda_{l-1}-\lambda_{l-2})...T_{m,r}^\prime(\lambda_{m+1}-\lambda_{m})T_{m-1,r}^\prime(\lambda_{m}-s)
\end{align}
where $T_{k,r}$ and $T_{k,r}^*$ are the $C_0$-semigroups associated with \begin{equation}\label{eq1 proof Thm equalities: adjoint EVF and EVF} \fra_{k,r}(u,v):=\frac{1}{\lambda_{k+1}-\lambda_k}\int_{\lambda_k}^{\lambda_{k+1}}\fra(T-r;u,v){\rm  d}r=\frac{1}{\lambda_{k+1}-\lambda_k}\int_{T-\lambda_{k+1}}^{T-\lambda_k}\fra(r;u,v){\rm  d}r\end{equation}
and its adjoint $\fra_{k,r}^*$, respectively. Recall that $T_{k,r}^*=T_{k,r}^\prime.$

\par\noindent On the other hand, the last equality in (\ref{eq1 proof Thm equalities: adjoint EVF and EVF}) implies that $T_{k,r}$ coincides with the semigroup associated with $\fra_{k,\Lambda_T}$ where $\Lambda_T$ is the subdivision $\Lambda_T:=(0=T-\lambda_{n+1}<T-\lambda_n<...<T-\lambda_1<T-\lambda_0=T).$ It follows from \eqref{promenade1}-\eqref{promenade2} and \eqref{eq1 proof returned adjoint}-\eqref{eq2 proof returned adjoint} that 
\begin{align*}\Big[\U^*_{\Lambda,r}(t,s)\Big]^\prime&=\T_{m-1,r}(\lambda_{m}-s)\T_{m,r}(\lambda_{m+1}-\lambda_{m})...\T_{l-2,r}(\lambda_{l-1}-\lambda_{l-2})\T_{l-1,r}(t-\lambda_{l-1})
\\&=\T_{m-1}\Big((T-s)-(T-\lambda_{m})\Big)\T_{m}\Big((T-\lambda_m)-(T-\lambda_{m+1})\Big)...\T_{l-1}\Big((T-\lambda_{l-1})-(T-t)\Big)\\&=\U_{\Lambda_T}(T-s,T-t)
\end{align*}
Finally,  the desired equality \eqref{Key equalities: adjoint EVF and EVF} follows by passing to the limit as $|\Lambda|=|\Lambda_T|\to 0.$
\end{proof}
\begin{remark}
	\label{remark-rescaling} The coerciveness assumption in  \eqref{eq:continuity-nonaut} may be replaced with 
	\begin{equation}\label{eq:Ellipticity-nonaut2}
	\Re \fra (t,u,u) +\omega\Vert u\Vert_H^2\ge \alpha \|u\|^2_V \quad ( t\in [0,T], u\in V)
	\end{equation}
	for some $\omega\in \R.$ In fact, $\fra$ satisfies (\ref{eq:Ellipticity-nonaut2}) if and only if the form $a_\omega$ given by $\fra_\omega(t;\cdot,\cdot):=\fra(t;\cdot,\cdot)+\omega (\cdot\mid \cdot)$ 
	satisfies the second inequality in \eqref{eq:continuity-nonaut}. Moreover, if $u\in MR(V,V')$ and $v:=e^{-w.}u,$ then $v\in MR(V,V')$ and $u$ satisfies (\ref{evolution equation u(s)=x}) if and only if $v$ satisfies 
	\begin{equation*}
	\dot{v}(t)+(\omega+\mathcal A(t))v(t)=0 \ \ \  t{\rm
		-a.e.}  \hbox{ on} \ [s,T],\
	\ \ \ \ v(s)=x. \
	\end{equation*}

\end{remark}
\section{Norm continuous evolution family}\label{Sec2 Norm continuity}

  In this section we assume  that the non-autonomous form $\fra$ satisfies
  (\ref{eq:continuity-nonaut})-\eqref{square property}. Thus as mentioned in the introduction, 
  under theses assumptions the Cauchy problem (\ref{evolution equation u(s)=x})
 has $L^2$-maximal regularity in $H$. 
 Thus for each $x\in V,$  \[ U(\cdot,s)x\in \MR(V,H):=\MR (s,T;V,H):=L^2(s,T;V)\cap H^1(s,T;H).\]
 Moreover, $ U(\cdot,s)x\in C[s,T];V)$ by \cite[Theorem 4.2]{Ar-Mo15}. From \cite[Theorem 2.7]{LH17} we known that the restriction of $\U$ to $V$ defines an evolution family which norm continuous. The same is also true for the Cauchy problem
(\ref{evolution equation u(s)=x returned}) and the assocaited evolution 
family $U^*_r$ since the returned adjoint form $\fra_r^*$ inherits all 
 properties of $\fra.$

\noindent  In the following we establish that $\U$ can be extended to a strongly continuous evolution
family on $V'.$ 
 \begin{proposition}\label{Lemma EVF on V'} Let $\fra$ be a non-autonomous sesquilinear form satisfying 
(\ref{eq:continuity-nonaut})-(\ref{square property}).
 Then $\U$ can be extended to a strongly continuous evolution family on $V^\prime,$ which we still denote $\U.$
 \end{proposition}
\begin{proof} Let $x\in H$ and $(t,s)\in\overline {\Delta}.$ Then using Proposition \ref{equalities: adjoint EVF and EVF} and the fact that $\U$ and $\U_r^*$ define both strongly continuous evolution families on $V$ and $H$ we obtain that 
\begin{align*} 
\|\U(t,s)x\|_{V'}&=\sup_{\underset{v\in V}{\|v\|_V=1}}\mid<\U(t,s)x, v>\mid
\\&=\sup_{\underset{v\in V}{\|v\|_V=1}}\mid (\U(t,s)x|v)_H\mid
=\sup_{\underset{v\in V}{\|v\|_V=1}}\mid (x| \U(t,s)^\prime v)_H\mid
\\&=\sup_{\underset{v\in V}{\|v\|_V=1}}\mid (x|\U_r^*(T-s,T-t)v)_H\mid
\\&=\sup_{\underset{v\in V}{\|v\|_V=1}}\mid <x,\U_r^*(T-s,T-t)v>\mid
\\&\leq \|x\|_{V^\prime}\|\U_r^*(T-s,T-t)\|_{\L(V)}
\\&\leq c\|x\|_{V^\prime}
\end{align*}
where $c>0$ is such that $\underset{t,s\in\Delta}{\sup}\|\U_r^*(t,s)\|_{\L(V)}\leq c.$ Thus, the claim follows since $H$ is dense in $V'.$
\end{proof}

 %

Let $\Delta:=\{(t,s)\in\Delta\mid t\ge s\}.$ The following theorem is the main result of this paper

\begin{theorem}\label{main result}
Let $\fra$ be a non-autonomous sesquilinear form satisfying 
(\ref{eq:continuity-nonaut})-(\ref{square property}). Let $\{U(t,s):\ (t,s)\in\Delta\}$ given by (\ref{evolution family}). Then the function $(t,s)\mapsto \U(t,s)$ is norm continuous on $\Delta$ into $\L(X)$ for $X=V, H$ and $V'.$ 
\end{theorem}
\begin{proof}
	The norm continuity for $\U$ in the case where $X=V$
	follows from \cite[Theorem 2.7]{LH17}. On the other hand, applying \cite[Theorem 2.7]{LH17} 
	to $\fra_r^*$ we obtain that $\U_r^*$ is also norm continuous on $\Delta$ with values in $\L(V).$ 
	Using Proposition  \ref{equalities: adjoint EVF and EVF}, we obtain by similar arguments as in the proof 
	of Lemma \ref{Lemma EVF on V'} 
	\begin{equation}
	 \|\U(t,s)-\U(t',s')\|_{\L(V')}\leq \|\U_r^*(T-s,T-t)x-\U_r^*(T-s',T-t')x\|_{\L(V)}
	\end{equation}
for all $x\in V'$ and $(t,s), (t',s')\in \Delta.$ This implies that $\U$ is norm continuous on $\Delta$ with values in $\L(V').$
Finally, the norm continuity in $H$ follows then by interpolation.
\end{proof}

\section{Examples}\label{S application}
\noindent This section is devoted to some relevant examples illustrating the theory developed 
in the previous sections. We refer to \cite{Ar-Mo15} and \cite{Ou15} and the references therein for further examples.

\medskip

$(i)$ \textbf{Laplacian with time dependent Robin boundary conditions on exterior domain} Let $\Omega$ be a bounded domain of $\R^d$ with Lipschitz boundary $\Gamma.$ Denote by $\sigma$ the $(d-1)$-dimensional Hausdorff measure on $\Gamma.$ Let $\Omega_{ext}$ denote the exterior domain of $\Omega,$ i.e., $\Omega_{ext}:=\R^d\setminus\overline{\Omega}.$ Let $T>0$ and $\alpha>1/4.$ Let $\beta:[0,T]\times \Gamma \lra \R$
be a bounded measurable function such that 
\[|\beta(t,\xi)-\beta(t,\xi)|\leq c|t-s|^\alpha\]
for some constant $c>0$ and every $t,s\in [0,T], \xi\in \Gamma.$ We consider the from $\fra:[0,T]\times H^1(\Omega_{ext})\times H^1(\Omega_{ext})\lra \C$ defined by 
\[\fra(t;u,v):=\int_{\Omega_{ext}}\nabla u\cdot\nabla v {\rm d}\xi+\int_{\Omega_{ext}}\beta(t,\cdot)\restr{u}{\Gamma} \restr{\bar v}{\Gamma} {\rm d}\sigma \]
where $u\to \restr{u}{\Gamma}: H^1(\Omega_{ext}) \lra L^2(\Gamma,\sigma)$ is the trace operator which is bounded \cite[Theorem 5.36]{AdFou}. The operator $A(t)$ associated with $\fra(t;\cdot,\cdot)$ on $H:=L^2(\Omega_{ext})$ is minus the Laplacian with time dependent Robin boundary conditions 
\[\partial_\nu u(t)+\beta(t,\cdot)u=0\ \text{ on } \Gamma. \]
Here $\partial_\nu$ is the weak normal derivative. Thus the domain of $A(t)$ is the set 
\[D(A(t))=\Big\{ u\in H^1(\Omega_{ext}) \mid \Laplace u\in L^2(\Omega_{ext}), \partial_\nu u(t)+\beta(t,\cdot)\restr{u}{\Gamma}=0 \Big\} \]
and for $u\in D(A(t)), A(t)u:=-\Laplace u.$ 
Thus similarly as in \cite[Section 5]{Ar-Mo15} one obtains that $\fra$ satisfies  (\ref{eq:continuity-nonaut})-(\ref{square property}) with $\gamma:=r_0+1/2$ and $\omega(t)=t^\alpha$ where $r_0\in(0,1/2)$ such that $r_0+1/2<2\alpha.$ We note that \cite[Section 5]{Ar-Mo15} the authors considered the Robin Laplacian on the bounded Lipschitz domain $\Omega.$ The main ingredient used there is  that the trace operators  are bounded from $H^{s}(\Omega)$ with value in $H^{s-1/2}(\Gamma,\sigma)$ for all $1/2<s<3/4.$ This boundary trace embedding theorem holds also for unbounded Lipschitz domain, and thus for $\Omega_{ext}$,  see \cite[Theorem 3.38]{Mclean} or \cite[Lemma 3.6]{Cos}.

Thus applying \cite[Theorem 4.1]{Ar-Mo15} and Theorem \ref{main result} we obtain that the non-autonomous Cauchy problem 

\begin{equation}\label{RobinLpalacian}
 \left\{
 \begin{aligned}
 \dot {u}(t) - \Laplace u(t)& = 0, \ u(0)=x\in H^1(\Omega_{ext })
 \\  \partial_\nu u(t)+\beta(t,\cdot){u}&=0 \ \text{ on } \Gamma
 \end{aligned} \right.
 \end{equation}
 has $L^2$-maximal regularity in $L^2(\Omega_{ext})$ and its solution is governed by an  evolution family $\U(\cdot,\cdot)$ that is norm continuous on each space $V, L^2(\Omega_{ext})$ and $V'.$  
 \subsection{Non-autonomous Schr\"odinger operators} 
 Let $m_0, m_1\in L_{Loc}^1(\R^d)$ and $m:[0,T]\times\R^d\lra \R$ be a measurable function such that there exist positive constants $\alpha_1,\alpha_2$ and $\kappa$ such that 
 \[\alpha_1 m_0(\xi)\leq m(t,\xi)\leq \alpha_2 m_0(\xi),
 \quad \text{ }\
 \quad \mid m(t,\xi)-m(s,\xi)\mid \leq \kappa|t-s|m_1(\xi)\]
for almost every $\xi\in \R^d$ and every $t,s\in [0,T].$ Assume moreover that there exist a constant $c>0$ and $s\in[0,1]$ such that  for $u\in C_c^\infty(\R^d)$
\begin{equation}\label{additional AS Svhrödinger example}\int_{\R^d}m_1(\xi)|u(\xi)|^2 \d\xi\leq c\|u\|_{H^s(\R^d)}. 
\end{equation}Consider the non-autonomous Cauchy problem 
\begin{equation}\label{Schroedinger operator}
\left\{
\begin{aligned}
\dot {u}(t) - &\Laplace u(t)+m(t,\cdot)u(t) = 0,
\\  u(0)&=x\in V. 
\end{aligned} \right.
\end{equation}
Here $A(t)=-\Laplace+m(t,\cdot)$ is associated with the non-autonomous form $\fra:[0,T]\times V\times V\lra \C$ given by
\[V:=\left\{u\in H^1(\R^d):  \int_{\R^d}m_0(\xi)|u(\xi)|^2  \d\xi <\infty \right\}\]and
\[\fra(t;u,v)=\int_{\R^d}\nabla u\cdot\nabla v  {\rm d} \xi+\int_{\R^d}m(t,\xi)|u(\xi)|^2 \d\xi. \] The form $\fra$ satisfies also (\ref{eq:continuity-nonaut})-(\ref{square property}) with $\gamma:=s$ and $\omega(t)=t^\alpha$ for $\alpha>\frac{s}{2}$ and $s\in[0,1].$ 

 This example is taken from \cite[Example 3.1]{Ou15}. 
Using our Theorem \ref{main result} we have that the solution of Cauchy problem (\ref{Schroedinger operator}) is governed by a norm continuous evolution family on $L^2(\R^d), V$ and $V'.$  
 


\begin{thebibliography}{999}
	
 \bibitem{Ac88} P.~Acquistapace. Evolution operators and strong solutions of abstract linear parabolic equations. \textit{ Differential Integral Equations} 1 (1988), no. 4, 433-457.
\bibitem{AdFou} R. A. Adams, J. J. F. Fournier. \textit{Sobolev spaces. }Second edition. Pure and Applied Mathematics (Amsterdam), 140. Elsevier/Academic Press, Amsterdam, 2003.
\bibitem{Ar06}  W.\ Arendt. \textit{Heat kernels.} $9^{th}$ Internet Seminar (ISEM) 2005/2006. Available at. https://www.uni-ulm.de/mawi/iaa/members/professors/arendt.html
	
	
	

	
	\bibitem{Ar-Mo15} W.~Arendt, S.~Monniaux. Maximal regularity for non-autonomous Robin boundary conditions. Math. Nachr. 1-16(2016) /DOI: 10.1002/mana.201400319
	
	

	\bibitem{Bre11} H.\ Br\'ezis. \emph{Functional Analysis, Sobolev Spaces and Partial Differential Equations}. Springer, Berlin 2011.
\bibitem{Cos} M. Costabel. Boundary integral operators on Lipschitz domains: elementary results. \textit{SIAM J. Math. Anal.} 19 (1988), no. 3, 613-626. 
	\bibitem{Daners01} D. Daners. Heat kernel estimates for operators with boundary conditions. Math. Nachr. 217 (2000), 13-41. 
	\bibitem{DL88} R.\ Dautray and J.L.\  Lions. \textit{Analyse Math\'ematique et
		Calcul Num\'erique pour les Sciences et les Techniques.} Vol.\ 8,
	Masson, Paris, 1988.
	
	
	\bibitem{ELLA13} O.~El-Mennaoui, H.~Laasri. {\it Stability for non-autonomous linear evolution equations with $L^p-$ maximal regularity.} \textit{ Czechoslovak Mathematical Journal.} 63 (138) 2013.
	
	\bibitem{ELLA15} O.~El-Mennaoui, H.~Laasri. On evolution equations governed by non-autonomous forms. Archiv der Mathematik (2016), 1-15, DOI 10.1007/s00013-016-0903-5 
	\bibitem{ENNA} K. J. Engel, R. Nagel. One-Parameter Semigroups for Linear Evolution Equations, Springer-Verlag, 2000.
	\bibitem{Fa17}{ S. Fackler} J.-L. Lions' problem concerning maximal regularity of equations governed by non-autonomous forms. Ann. Inst. H. Poincaré Anal. Non Linéaire 34 (2017)
	\bibitem{Ka} T. Kato. \textit{Perturbation theory for linear operators.} Springer-Verlag, Berlin 1992.
	
	\bibitem{H.Ko61} H. Komatsu, Abstract analyticity in time and unique continuation property of solutions of
	a parabolic equation, J. Fac. Sci. Univ. Tokyo, Sect. 1 9 (1961), 1-11.
	\bibitem{LH17} H. Laasri. Regularity properties for evolution family governed by non-autonomous forms. Archiv der Mathematik (2018), https://doi.org/10.1007/s00013-018-1175-z.
\bibitem{Lio61} J.L.\ Lions. \textit{ Equations Diff\'erentielles Op\'erationnelles
		et Probl\`emes aux Limites.} Springer-Verlag, Berlin, G\"ottingen, Heidelberg, 1961.
	\bibitem{Lu89} A. Lunardi. Differentiability with respect to $(t,s)$ of the parabolic evolution operator.
	Israel J. Math. 68 (1989), no. 2, 161-184. 
\bibitem{Mclean} W. Mclean. \textit{Strongly elliptic systems and boundary integral equations.} Cambridge University Press, Cambridge, 2000.
	\bibitem{Ou15} E. M.\ Ouhabaz. \textit{Maximal regularity for non-autonomous evolution equations governed by forms having less regularity.} Arch. Math. 105 (2015), 79-91.
	
	Princeton Univ.\ Press 2005. 
	\bibitem{LASA14} A. Sani, H. Laasri, { Evolution Equations governed by Lipschitz Continuous Non-autonomous Forms.} \textit{Czechoslovak Mathematical Journal.} 65 (140) (2015), 475-491.
	\bibitem{Sho97} R.\ E.\ Showalter.\textit{ Monotone Operators in Banach
		Space and Nonlinear Partial Differential Equations.} Mathematical
	Surveys and Monographs.  American Mathematical Society, Providence,
	RI, 1997.
	

	
	\bibitem{Tan79} H. Tanabe. \textit{Equations of Evolution.} Pitman 1979.
	
	
	
\end{thebibliography}
\end{document}